\documentclass[12pt]{amsart}

\usepackage{amsmath} \usepackage{amssymb} \usepackage{amsthm}
\usepackage{amscd}
\usepackage[all]{xy}
\usepackage{enumerate}
\usepackage{mathrsfs}
\usepackage{multirow}
\usepackage{bigdelim}
\usepackage{color}
\usepackage[bookmarks=false]{hyperref}

\def\tb{\textbf}

\def\mb{\mathbb}

\def\a{{\alpha}}

\def\E{{\mathcal{E}}}
\def\F{{\mathcal{F}}}
\def\G{{\mathcal{G}}}

\def\L{{\mathcal{L}}}
\def\M{{\mathcal{M}}}
\def\o{{\omega}}
\def\O{{\mathcal{O}}}

\def\Z{{\mathbb{Z}}}

\def\Q{{\mathbb{Q}}}

\def\ol{\overline}

\def\Pic{{\mathrm{Pic}}}

\def\rank{{\mathrm{rank}}}

\theoremstyle{plain}
\newtheorem{thm}{Theorem}[section]

\newtheorem{prop}[thm]{Proposition}
\newtheorem{conj}[thm]{Conjecture}

\newtheorem{lem}[thm]{Lemma}
\theoremstyle{definition}

\theoremstyle{remark}

\newtheorem*{cl}{Claim}

\newtheorem*{acknowledgement}{Acknowledgments}

\title{Iitaka's $C_{n,m}$ conjecture for 3-folds in positive characteristic}
\author{Sho Ejiri}
\address{Sho Ejiri\\Department of Mathematics, Graduate School of Science, Osaka University, Toyonaka, Osaka 560-0043, Japan}
\email{s-ejiri@cr.math.sci.osaka-u.ac.jp}
\author{Lei Zhang}
\address{Lei Zhang\\School of Mathematical Science, University of Science and Technology of China, Hefei 230026, P.R.China.}
\email{zhleimath@163.com, lzhpkutju@gmail.com}

\begin{document}
\tolerance = 9999

\maketitle
\markboth{SHO EJIRI AND LEI ZHANG}{IITAKA'S CONJECTURE FOR 3-FOLDS IN CHARACTERISTIC $p>5$}
\begin{abstract}
In this paper, we prove that for a fibration $f:X\to Z$ from a smooth projective 3-fold to a smooth projective curve, over an algebraically closed field $k$ with $\mathrm{char} k =p >5$, if the geometric generic fiber $X_{\ol\eta}$ is smooth, then subadditivity of Kodaira dimensions holds, i.e.
$$\kappa(X)\ge\kappa(X_{\ol\eta})+\kappa(Z).$$
\end{abstract}
\section{Introduction} \label{section:intro}

Throughout this paper, a {\it fibration} means a projective morphism $f:X\to Y$
between varieties such that the natural morphism $\O_Y \to f_*\O_X$ is an isomorphism.

The Kodaira dimension is one of the most important birational invariants and plays a key role in the birational classification of algebraic varieties.
For a fibration, we have the following conjecture on Kodaira dimensions, which was proposed by Iitaka in characteristic zero.
\begin{conj}[\textup{$C_{n,m}$}]\label{conj:Iitaka}\samepage
Let $f:X\to Z$ be a fibration between smooth projective varieties of dimension $n$ and $m$ respectively over an algebraically closed field $k$ with $\mathrm{char} k = p \ge 0$ , whose geometric generic fiber $X_{\ol\eta}$ is integral and smooth.
Then
$$\kappa(X)\ge\kappa(X_{\ol\eta})+\kappa(Z).$$
\end{conj}
In characteristic zero, many results related to this conjecture are known \cite{Bir09,Cao15,CP15,CH11,Fuj13,Fuj14,Kaw81,Kaw82,Kaw85,Kol87,KP15,Lai11,Vie77,Vie83}. In particular, this conjecture was reduced to problems in the minimal model program by Kawamata \cite[Corollary 1.2]{Kaw85}. \par
In positive characteristic, Conjecture \ref{conj:Iitaka} has been proved in some cases recently. Chen and Zhang showed $C_{n,n-1}$ \cite[Theorem 1.2]{CZ13}. Under the assumption that $p>5$, $C_{3,1}$ was shown when $k=\ol{\mb F_p}$ by Birkar, Chen and Zhang \cite[Theorem 1.2]{BCZ15}, when $X_{\ol\eta}$ is of general type \cite[Theorem 1.5]{Eji15} (see also \cite[Appendix 7]{Zha16}), and when the genus of $Z$ is at least two by Zhang \cite[Corollary 1.9]{Zha16}. Furthermore, when $f$ has singular geometric generic fiber, its dualizing sheaf, denoted by $\omega_{X_{\ol\eta}}$ (the same notation with canonical sheaf because they coincide with each other when $X_{\ol\eta}$ is smooth), was considered by \cite{Pat13} and \cite{Zha16}, and under some special situations, an analogous inequality $\kappa(X)\ge\kappa(X_{\ol\eta}, \omega_{X_{\ol\eta}})+\kappa(Z)$
was proved.

The aim of this paper is to prove the theorem below.
\begin{thm}\label{thm:c31_intro}\samepage
Conjecture $C_{3,m}$ holds when $\mathrm{char} k = p>5$.
\end{thm}
The proof relies on the minimal model program for varieties of dimension at most three in characteristic $p>5$ developed by several mathematicians including Birkar, Cascini, Hacon, Tanaka, Waldron and Xu. The case when $\kappa(X_{\ol\eta}) = 2$ has been proved by the first author \cite[Theorem 1.5]{Eji15}. In this paper we only need to consider the cases $\kappa(X_{\ol\eta}) = 0, 1$.

This paper is organized as follows. Section \ref{section:mmp} includes some basic results to be used in our proof, including minimal model theory of 3-folds, vector bundles on elliptic curves and weak positivity of push-forward of pluri-relative canonical sheaves. Section \ref{section:k(F)=0} and \ref{section:k(F)=1} are devoted to study the cases $\kappa(X_{\ol\eta}) = 0$ and $1$ respectively.

\textbf{Notation and Conventions:} In this paper, we fix an algebraically closed field $k$ of characteristic $p>0$.
A {\it $k$-scheme} is a separated scheme of finite type over $k$. A {\it variety}
means an integral $k$-scheme, and a {\it curve} (resp. {\it surface, $n$-fold}) means
a variety of dimension one (resp. two, $n$).

Let $\varphi:S\to T$ be a morphism of schemes and let $T'$ be a $T$-scheme.
Then we denote by $S_{T'}$ and $\varphi_{T'}:S_{T'}\to T'$ respectively the fiber product
$S\times_{T}T'$ and its second projection. For a prime $p\in\Z$, $\mb F_p$ and $\Z_{(p)}$
denote respectively $\Z/p\Z$ and the localization of $\Z$ at $p\Z$. For a Cartier, $\Z_{(p)}$-Cartier
or $\Q$-Cartier divisor $D$ on $S$ (resp. an $\O_S$-module $\G$), the pullback of $D$ (resp. $\G$) to $S_{T'}$
is denoted by $D_{T'}$ or $D|_{S_{T'}}$ (resp. $\G_{T'}$ or $\G|_{S_{T'}}$) if it is well-defined. Similarly, for a homomorphism of $\O_S$-modules
$\alpha:\F\to\G$ the pullback of $\alpha$ to $S_{T'}$ is denoted by $\a_{T'}:\F_{T'}\to \G_{T'}$.

\begin{small}
\begin{acknowledgement}
The first author is greatly indebted to his supervisor Professor Shunsuke Takagi for suggesting the problems in this paper, for many fruitful discussions and for much helpful advice.
He is deeply grateful to Professors Caucher Birkar and Yifei Chen for valuable comments and suggestions.
He would like to thank Professors Yoshinori Gongyo and Akiyoshi Sannai for helpful comments.
He also would like to thank Doctors Takeru Fukuoka and Hokuto Konno for useful comments.
He is supported by JSPS KAKENHI Grant Number 15J09117 and the Program for Leading Graduate Schools, MEXT, Japan. Part of this work was done during the period of the second author visiting BICMR, he would like to thank Prof. Chenyang Xu for his hospitality and useful discussion. The second author is supported by grant
NSFC (No. 11401358 and No. 11531009).
The authors thank the referee for many useful comments and suggestions to improve this paper.
\end{acknowledgement}
\end{small}

\section{Preliminaries}\label{section:mmp}
In this section, we recall some basic results which will be used in the proof.
\subsection{Minimal models of $3$-folds}
Existence of (log) minimal models of $3$-folds in positive characteristic $p>5$ was first proved for canonical singularities by Hacon and Xu \cite{HX15}, and in general by Birkar \cite{Bir13} (see \cite{Wal16} for the lc case).
The result on Mori fiber spaces was proved for terminal singularities by Cascini, Tanaka and Xu \cite{CTX15},
and in general by Birkar and Waldron \cite{BW14}. We collect some results in the following theorem, which will be used in our proof.
\begin{thm}\label{thm:mmp}\samepage
Assume that the base field $k$ has characteristic $p>5$. Let $f:X\to Z$ be a contraction from a normal $3$-fold, and let $\Delta$ be an effective $\Q$-Cartier $\Q$-divisor on $X$.

(1) If either $(X,\Delta)$ is klt and  $K_X+\Delta$ is pseudo-effective over $Z$, or $(X,\Delta)$ is lc and $K_X+\Delta$ has a weak Zariski decomposition \footnote{i.e., there exists a birational projective morphism $\mu: W \to X$ such that $\nu^*(K_X + \Delta) = P + M$ where $P$ is nef over $Z$ and $M$ is effective},
then $(X,\Delta)$ has a log minimal model over $Z$.

(2) If $(X, \Delta)$ is a dlt pair and $Z$ is a smooth projective curve with $g(Z) \geq 1$, then every step of LMMP in \cite[Sec. 3.5-3.7]{Bir13} starting from $(X, \Delta)$ is over $Z$.
\end{thm}
\begin{proof}
For (1) please refer to \cite[Theorem 1.2 and Proposition 8.3]{Bir13}.

For (2), since $(X, \Delta)$ is dlt, every $K_X + \Delta$-extremal ray is a generated by a rational curve by cone theorem \cite[Theorem 1.1]{BW14}, which is contracted by $f$ since $g(Z) \geq 1$. So for an extremal contraction $X \to \bar{X}$, if there is a divisorial contraction or a flip $\sigma: X \dashrightarrow X^+$ as in \cite[Sec. 3.5-3.7]{Bir13},
there exist natural morphism $\bar{f}: \bar{X} \to Z$ and $f^+: X^+ \to Z$ fitting into the following commutative diagram
$$\xymatrix{&X \ar[rd]\ar[rdd]_f\ar@{.>}[rr] & &X^+ \ar[ld]\ar[ldd]^{f^+}\\
&&\bar{X} \ar[d]^{\bar{f}} &\\
 &&Z &. }$$
Note that $(X^+, \Delta^+ = \sigma_*\Delta)$ is a dlt pair. We can show this assertion by induction.
\end{proof}

\subsection{Covering Theorem}
The result below is [\cite{Iit82}, Theorem 10.5] when $X$ and $Y$ are both smooth, and the proof there also applies when varieties are normal.
\begin{thm}\textup{(\cite[Theorem 10.5]{Iit82})}\label{ct}
Let $f\colon X \rightarrow Y$ be a proper surjective morphism between complete normal varieties.
If $D$ is a Cartier divisor on $Y$ and $E$ an effective $f$-exceptional divisor on $X$, then
$$\kappa(X, f^*D + E) = \kappa(Y, D).$$
\end{thm}

As a corollary we get the following useful result.
\begin{lem}\label{injofpic}
Let $g: W \rightarrow Y$ be surjective projective morphism between projective varieties. Assume $Y$ is normal and let $L_1,L_2 \in \Pic^0(Y)$ be two line bundles on $Y$. If $g^*L_1 \sim_{\mathbb{Q}} g^*L_2$ then $L_1 \sim_{\mathbb{Q}} L_2$.
\end{lem}
\begin{proof}
Let $L = L_1 \otimes L_2^{-1}$. Denote by $\sigma: W' \to W$ the normalization and  $g'=g\circ \sigma: W' \to Y$. Then $g'^*L \sim_{\Q} 0$. Applying Theorem \ref{ct} to $g': W' \to Y$ gives that $L \sim_{\Q} 0$, which is equivalent to that $L_1 \sim_{\mathbb{Q}} L_2$.
\end{proof}

\subsection{Adjunction}
\begin{lem}\label{adj}
Assume that the base field $k$ has characteristic $p>5$. Let $(X, \Delta)$ be a normal, $\mathbb{Q}$-factorial, lc 3-fold (not necessarily projective). Let $C$ be a projective lc center of $(X, \Delta)$ and $\tilde{C}$ be the normalization of $C$. If $(K_X + \Delta)|_{\tilde{C}}$ is numerically trivial, then $(K_X + \Delta)|_{\tilde{C}}$ is $\mathbb{Q}$-trivial.
\end{lem}
\begin{proof}
By \cite[Lemma 6.5]{Bir13}, we can take a crepant partial resolution $\mu:X' \rightarrow X$ such that
$$K_{X'} + D + \Delta' \sim_{\mathbb{Q}} \mu^*(K_X + \Delta)\cdots (\clubsuit)$$
where $D$ is a reduced irreducible divisor dominant over $C$ and $(X', D + \Delta')$ is dlt. Then considering the restriction of the relation $\clubsuit$ on $D$, by the adjunction formula \cite[5.3]{Ke99}, we have
$$K_{D} + \Delta_D \sim_{\mathbb{Q}} \mu^*(K_X + \Delta)|_D.$$
Then $D$ is a normal projective surface hence granted a natural morphism $D \to \tilde{C}$, and $(D, \Delta_D)$ is log canonical by \cite[Lemma 5.2]{Bir13}. Applying \cite[Theorem 1.2]{Tan14}, we have that $K_{D} + \Delta_D$ is semi-ample, thus $\mu^*(K_X + \Delta)|_{D}$ is $\mathbb{Q}$-trivial since $(K_X + \Delta)|_{\tilde{C}}$ is numerically trivial. We can conclude that $(K_X + \Delta)|_{\tilde{C}}$ is $\mathbb{Q}$-trivial by Lemma \ref{injofpic}.
\end{proof}
\subsection{Vector bundles on elliptic curves}
In this subsection, we recall some facts about vector bundles on elliptic curves,
which are used in the proof of Theorem \ref{thm:sa_intro}.
\begin{thm}\label{thm:facts on vb on ell curve}\samepage Let $C$ be an elliptic curve, and let $\E_C(r,d)$ be the set of isomorphism classes of indecomposable vector bundles of rank $r$ and of degree $d$.
\begin{itemize}
\item[(1)]\textup{(\cite[Theorem 5]{Ati57})}For each $r>0$, there exists a unique element $\E_{r,0}$ of
$\E_C(r,0)$ with $H^0(C,\E_{r,0})\ne0$. Moreover, for every $\E\in\E_C(r,0)$ there exists an $\L\in\Pic^0(C)$ such that $\E\cong\E_{r,0}\otimes\L$.
\item[(2)]\textup{(\cite[Proposition 2.10]{Oda71})} When the Hasse invariant ${\rm Hasse}(C)$
is nonzero, $F_C^*\E_{r,0}\cong \E_{r,0}$. When ${\rm Hasse}(C)$ is zero,
$F_C^*\E_{r,0}\cong \bigoplus_{1\le i\le\min\{r,p\}}\E_{\lfloor(r-i)/p\rfloor+1,0}$,
where $\lfloor r\rfloor$ denotes the round down of $r$.
\end{itemize}
\end{thm}
\begin{thm}[\textup{\cite[1.4. Satz]{LS77}}]\label{thm:fact on vb on sm curve}\samepage
Let $\E$ be a vector bundle on a smooth projective curve $C$.
If ${F_C^e}^*\E\cong\E$ for some $e>0$, then there exists an \'etale morphism $\pi:C'\to C$ from a smooth
projective curve $C'$ such that $\pi^*\E\cong\bigoplus\O_{C'}$.
\end{thm}
\begin{prop}\label{prop:decomp}\samepage
Let $\E$ be a vector bundle on an elliptic curve $C$. Then there exists a finite morphism $\pi:C'\to C$ from an elliptic curve $C'$ such that $\pi^*\E$ is a direct sum of line bundles.
\end{prop}
\begin{proof}
We may assume that for every finite morphism $\varphi:B\to C$ from an elliptic curve $B$, $\varphi^*\E$ is indecomposable. Set $d:=\deg\E$ and $r:=\rank\E$. We show that $r=1$. Let $Q\in C$ be a closed point. Replacing $\E$ by $((r_C)^*\E)(-dQ)$, we may assume that $d=0$. Here $r_C:C\to C$ is the morphism given by multiplication by $r$. Hence Theorems \ref{thm:facts on vb on ell curve} and \ref{thm:fact on vb on sm curve} imply that when the Hasse invariant of $C$ is nonzero (resp. zero), there exists an \'etale morphism $\pi:C'\to C$ (resp. an $e>0$) such that $\pi^*\E$ (resp. ${F_C^e}^*\E$) is a direct sum of line bundles. This implies that  $r=1$.
\end{proof}
\subsection{Weak positivity}
The following positivity result will be used in the proof of the case when the geometric generic fiber has Kodaira dimension one.
\begin{thm}\label{mthp}
Assume that $\mathrm{char} k=p>5$.
Let $f: X \rightarrow Z$ be a fibration from a smooth projective 3-fold to a smooth projective curve.
Suppose that the geometric generic fiber $X_{\ol\eta}$ has at most rational double points as singularities.
If $\kappa(X_{\ol\eta},K_{X_{\ol\eta}})=1$, then there exists a real number $c>0$ such that
$f_*\omega_{X/Z}^m$ contains a nef subbundle of rank at least $cm$ for sufficiently divisible $m>0$.
\end{thm}
%
%
Before proving Theorem \ref{mthp}, we recall some results.
\begin{thm}[\textup{\cite[3.2]{CZ13}}]\label{thm:ellfib_intro}\samepage
Let $f:X\to Z$ be a surjective morphism between smooth projective varieties, over an algebraically closed field of positive characteristic, whose geometric generic fiber is a smooth elliptic curve. Then $\kappa(X,K_{X/Z})\ge0$.
\end{thm}
The following lemma will be frequently used.
\begin{lem}[\textup{\cite[Lemma 3.2]{Wal15}}] \label{l-linear-pullback}
Let $f: X \rightarrow Z$ be a fibration between normal quasi-projective varieties.  Let $L$ be a $f$-nef $\mathbb{Q}$-Cartier divisor on $X$ such that $L_{\eta}\sim_{\mathbb{Q}}0$ where $\eta$ is the generic point of $Z$.  Assume $\dim Z\le 3$. Then there exist a diagram $$ \xymatrix{ X'\ar[r]^\phi\ar[d]_{f'} & X\ar[d]^f\\ Z'\ar[r]^\psi & Z } $$ with $\phi,\psi$ projective birational, and an $\mathbb{Q}$-Cartier divisor $D$ on $Z'$ such that $\phi^* L\sim_{\mathbb{Q}} f'^*D$.
Furthermore, if $f$ is flat and $Z$ is $\Q$-factorial, then we can take $X'=X$ and $Z'=Z$.
\end{lem}
The next lemma is a consequence of Tanaka's vanishing theorem for surfaces \cite{Tan15x}.
\begin{lem}\label{mthplem}
Let $g:Y\to Z$ be a generically smooth surjective morphism from a smooth projective surface to a smooth projective curve. Let $H$ be a nef and $g$-big divisor on $Y$. Then $g_*\O_Y(K_{Y/Z}+lH)$ is a nef vector bundle for every $l\gg 0$.
\end{lem}
\begin{proof}
Let $A$ be an ample divisor on $Z$ with $\deg A \geq \deg K_Z + 2$.
Then $A - K_Z - z$ is ample for a closed point $z \in Z$ where $z$ is seen as a divisor on $Z$. Note that $\nu(H) \geq 1$ and $H+g^*(A-K_Z - z)$ is nef and big.
Denote by $Y_z$ the fiber of $g$ over $z$.
By \cite[Theorem 2.6]{Tan15x} we see that
$$H^1(Y,K_Y+H+g^*(A-K_Z) +(l-1)H - Y_z) = H^1(Y,K_Y+H+g^*(A-K_Z -z)  +(l-1)H)=0$$
for $l\gg0$.
Thus for a closed point $z\in Z$, by the long exact sequence arising from taking cohomology of the exact sequence below
$$0 \to \O_Y(K_{Y/Z}+g^*A +lH - Y_z) \to \O_Y(K_{Y/Z}+g^*A+lH) \to \O_Y(K_{Y/Z}+g^*A+lH)|_{Y_z} \to 0$$
we conclude that the restriction $$H^0(Y,K_{Y/Z}+g^*A+lH)\to H^0(Y_z,(K_{Y/Z}+g^*A+lH)|_{Y_z})$$ is surjective.
This implies that $(g_*\O_Y(K_{Y/Z}+lH))(A)$ is generically globally generated.
On the other hand, if $z$ is general then $Y_z$ is smooth, applying \cite[Corollary 2.23]{Pat14}, since $H|_{Y_z}$ is ample, we have that for $l\gg0$ the morphism
\begin{align*}
H^0(Y_{z},\phi^{(e)}_{Y_{z}}\otimes\O_{Y_{z}}(K_{Y_{z}}+lH_{z})):
H^0(Y_{z},K_{Y_{z}}+lp^eH_{z})\to H^0(Y_{z},K_{Y_{z}}+lH_{z})
\end{align*}
is surjective.
This implies that the homomorphism (\cite[Section 2]{Eji15})
\begin{align*}
{g_{Z^e}}_*&(\phi^{(e)}_{Y/Z}\otimes\O_{Y_{Z^e}}((K_{Y/Z}+lH)_{Z^e}))\otimes \O_{Z^e}(A):\\
&g_*\O_{Y}(K_{Y/Z}+lp^eH + g^*(A+z)) \cong g_*\O_{Y}(K_{Y/Z}+lp^eH)\otimes \O_{Z^e}(A) \\
&\to {g_{Z^e}}_*\O_{Y_{Z^e}}((K_{Y/Z}+lH)_{Z^e})\otimes \O_{Z^e}(A+z) \cong {F_Z^{e}}^*g_*\O_Y(K_{Y/Z}+lH)\otimes \O_{Z^e}(A)
\end{align*}
is generically surjective.
Thus for every $e>0$, ${F_Z^e}^*(g_*\O_Y(K_{Y/Z}+lH))\otimes \O_{Z^e}(A)$ is generically globally generated, and hence is nef.
We conclude that $g_*\O_Y(K_Y+lH)$ is nef by applying \cite[Proposition 4.7]{Eji15}.
\end{proof}
\begin{proof}[Proof of Theorem \ref{mthp}]
Let $W$ be a minimal model of $X$ over $Z$. Let $\rho:X_{\ol\eta}\to W_{\ol\eta}$ be the induced morphism. Since $\rho_*\O_{X_{\ol\eta}}\cong\O_{W_{\ol\eta}}$, $W_{\ol\eta}$ is normal. Furthermore, since $W$ is terminal, we have $K_{X_{\ol\eta}}\ge \rho^*K_{W_{\ol\eta}}$, and hence $W_{\ol\eta}$ has at most canonical singularities. In particular, replacing $X$ with a minimal model, with the loss of smoothness we may assume that $K_{X/Z}$ is $f$-nef.

Then by \cite[Theorem 1.2]{Tan14}, $K_{X_{\ol\eta}}$ is semi-ample, and since $p>5$, the geometric generic fiber of the Iitaka fibration $I_{\ol\eta}: X_{\ol\eta} \to C_{\ol\eta}$ is a smooth elliptic curve over $k(\ol\eta)$ by \cite[Theorem 7.18]{Bad01}. For the generic fiber $X_{\eta}$ and sufficiently divisible positive integer $n$, since $H^0(X_{\ol\eta}, nK_{X_{\ol\eta}}) \cong H^0(X_{\eta}, nK_{X_{\eta}})\otimes_{k(\eta)} k(\ol\eta)$, we see that the Iitaka fibration $I_{\ol\eta}: X_{\ol\eta} \to C_{\ol\eta}$ coincides with the Iitaka fibration $I_{\eta}: X_{\eta} \to C_{\eta}$ tensoring with $k(\ol\eta)$. Thus the the geometric generic fiber of $I_{\eta}$ is a smooth elliptic curve.

Considering the relative Iitaka fibration of $f: X \to Z$, whose geometric generic fiber is a smooth elliptic curve, we get a birational morphism $u:X'\to X$, a fibration $g:Y\to Z$ with $Y$ smooth, and an elliptic fibration $h:X'\to Y$ fitting into the following commutative diagram:
$$\xymatrix{&X'\ar[d]_{h}\ar[r]^{u} &X \ar[d]^{f}\\ &Y\ar[r]^{g} &Z. }$$
Note that the geometric generic fiber $C_{\ol\eta}$ of $g:Y\to Z$ is normal, and hence smooth.
By Lemma \ref{l-linear-pullback}, we may assume that $u^*K_{X/Z}\sim_{\Q}h^*H$ for a nef $g$-big $\Q$-Cartier divisor on $Y$.
By Theorem \ref{thm:ellfib_intro}, we have $\kappa(X',K_{X'/Y})\ge0$, and hence there exists an injective homomorphism $h^*\o_{Y/Z}^m\to\o_{X'/Z}^m$ for sufficiently divisible $m>0$.
Let $l\gg0$ be an integer such that $lH$ is Cartier and $u^*lK_{X/Z}\sim h^*lH$. Then we have natural homomorphisms
\begin{align*}
(g_*\O_Y(K_{Y/Z}+lH))^{\otimes m}&\to g_*\O_Y(m(K_{Y/Z}+lH))\cong g_*h_*\O_{X'}(mh^*(K_{Y/Z}+lH))\\
&\hookrightarrow f_*u_*\O_{X'}(mK_{X'/Z}+u^*lmK_{X/Z})\cong f_*\O_{X}(m(l+1)K_{X/Z}).
\end{align*}
Replacing $l$ if necessary, we may assume that the first homomorphism is generically surjective.
By Lemma \ref{mthplem}, $g_*\O_Y(K_{Y/Z}+lH)$ is nef, and hence so is $g_*\O_Y(m(K_{Y/Z}+lH))$. This completes the proof.
\end{proof}
\section{The case $\kappa(X_{\ol\eta})=0$}\label{section:k(F)=0}
In this section, we prove Theorem \ref{thm:c31_intro} in the case when the Kodaira dimension of the geometric generic fiber is equal to zero. It is proved as a consequence of Theorems \ref{thm:ps-eff_intro} and \ref{thm:sa_intro}.
\begin{thm}[\textup{\cite{Eji16}}]\label{thm:ps-eff_intro}\samepage
Let $f:X\to Z$ be a surjective morphism from a normal projective variety $X$ 
over an algebraically closed field of characteristic $p>0$ 
to a smooth projective variety $Z$,  
and let $\Delta$ be an effective $\Q$-divisor on $X$ 
such that $a\Delta$ is integral for some $a>0$ not divisible by $p$. 
Assume that $(X_{\ol\eta},\Delta_{\ol\eta})$ is $F$-pure, where $\ol\eta$ is the geometric generic point of $Z$. 
If $K_X+\Delta\sim_{\Q}f^*(K_Z+L)$ for some $\Q$-divisor $L$ on $Z$, then $L$ is pseudo-effective.
\end{thm}
Theorem \ref{thm:ps-eff_intro} follows from \cite[Theorem 4.5]{Eji16} (by setting $D=-(K_Z+L)$), 
and it is also proved by Patakfalvi \cite[Theorem 1.6]{Pat14} when $Z$ is a curve.
\begin{thm}\label{thm:sa_intro}\samepage
With the same notation and assumptions as in Theorem \ref{thm:ps-eff_intro}, if $Z$ is an elliptic curve, then $L$ is semi-ample.
\end{thm}
\begin{proof}
By Theorem \ref{thm:ps-eff_intro}, we have $\deg L\ge0$.
We may assume that $\deg L=0$, and it suffices to show that $L\sim_{\Q}0$.
Since $(K_X+\Delta)_{\ol\eta}\sim_{\Q}0$, there is an ample Cartier divisor $A$ on $X$ such that $l(K_X+\Delta)_{\ol\eta}+A_{\ol\eta}$ is ample and free for every $l\in a\Z$. Recall that $0<a\in \Z\setminus p\Z$ and $a\Delta$ is integral.
By Fujita's vanishing theorem, there exist some $m_0>0$ such that for every nef Cartier divisor $N$ on $X_{\ol\eta}$,
$\O_{X_{\ol\eta}}((m_0-1)A_{\ol\eta}+N)$ is $0$-regular with respect to $l(K_X+\Delta)_{\ol\eta}+A_{\ol\eta}$ for every $l\in a\Z$.
Then the natural homomorphism
\begin{align*}
H^0(X_{\ol\eta},l(K_X+\Delta)_{\ol\eta}+mA_{\ol\eta})&\otimes H^0(X_{\ol\eta},(m'-1)A_{\ol\eta})\otimes H^0(X_{\ol\eta},l'(K_X+\Delta)_{\ol\eta}+A_{\ol\eta})\\
&\to H^0(X_{\ol\eta},(l+l')(K_X+\Delta)_{\ol\eta}+(m+m')A_{\ol\eta})
\end{align*}
is surjective for every $l,l'\in a\Z$ and $m,m'\ge m_0$. Thus
\begin{align*}
H^0(X_{\ol\eta},l(K_X+\Delta)_{\ol\eta}+mA_{\ol\eta})&\otimes H^0(X_{\ol\eta},l'(K_X+\Delta)_{\ol\eta}+m'A_{\ol\eta})\\
&\to H^0(X_{\ol\eta},(l+l')(K_X+\Delta)_{\ol\eta}+(m+m')A_{\ol\eta})
\end{align*}
is also surjective, and hence the natural homomorphism $$\G(l,m)\otimes\G(l',m')\to\G(l+l',m+m')$$ is generically surjective, where $\G(l,m):=f_*\O_X(l(K_{X/Z}+\Delta)+mA)$.
From now on, we use the same notation as \cite[Sections 2 or 3]{Eji15} or \cite[Section 2]{Eji16}.
Replacing $m_0$ if necessary, by \cite[Corollary 2.23]{Pat14} we may assume that
\begin{align*}
H^0(&X_{\ol\eta},\phi_{(X_{\ol\eta},\Delta_{\ol\eta})}^{(e)}\otimes\O_{X_{\ol\eta}}(N+m_0A_{\ol\eta})):\\
&H^0(X_{\ol\eta},(1-p^e)(K_X+\Delta)_{\ol\eta}+p^e(N+m_0A_{\ol\eta})) \to H^0(X_{\ol\eta},N+m_0A_{\ol\eta})
\end{align*}
is surjective for every $e>0$ with $a|(p^e-1)$ and for every nef Cartier divisor $N$ on $X_{\ol\eta}$. Since $l(K_X+\Delta)_{\ol\eta}$ is nef,
\begin{align*}
{f_{Z^e}}_*&(\phi_{(X/Z,\Delta)}^{(e)}\otimes\O_{X_{Z^e}}(l(K_{X/Z}+\Delta)_{Z^e}+m_0A_{Z^e})):\\
& \G((l-1)p^e+1,m_0p^e) \to {f_{Z^e}}_*\O_{X_{Z^e}}(l(K_{X/Z}+\Delta)_{Z^e}+m_0A_{Z^e})\cong {F_Z^e}^*\G(l,m_0)
\end{align*}
is generically surjective.
Let $b>0$ be an integer such that $a|b$, that $bL$ is integral and that $b(K_X+\Delta)$ is linearly equivalent to $bf^*L$.
By Proposition \ref{prop:decomp}, there exists a finite morphism $\pi:Z'\to Z$ from an elliptic curve $Z'$ such that $\pi^*\G(r,m_0)$ is a direct sum of line bundles for each $0\le r<b$ with $a|r$.
By Lemma \ref{injofpic}, we may replace $L$ and $\G(r,m_0)$ respectively by its pullback by $\pi$.
Set \begin{align*} \F:=&\bigoplus_{0\le r<b,~a|r}\G(r,m_0),\\
\mu:=&\min\{\deg\M|\textup{$\M\in\Pic(Z)$ and $\M$ is a direct summand of $\F$}\}\textup{, and} \\
T:=&\{\M\in\Pic(Z)|\textup{$\deg\M=\mu$ and $\M$ is a direct summand of $\F$}\}\\
=&\{\M_1,\ldots,\M_\lambda\}.
\end{align*}
Then for every $\M_i\in T$, there exists an $0\le s<b$ with $a|s$ such that the composition
\begin{align*}
\G(s,m_0)^{\otimes p^e-1}&\otimes\G(r_{i,e},m_0)\otimes\O_Z(-q_{i,e}bL)\\
&\to\G((s-1)p^e+1,p^em_0)\to{F_Z^e}^*\G(s,m_0)\twoheadrightarrow\M_i^{p^e}
\end{align*}
is generically surjective for every $e>0$ with $a|(p^e-1)$.
Here $q_{i,e}$ and $r_{i,e}$ are integers satisfying $1+s-p^e=-q_{i,e}b+r_{i,e}$ and $0\le r_{i,e}<b$.
Then there exist a line bundle $\M$ which is a direct summand of $\G(s,m_0)^{p^e-1}\otimes\G(r_{i,e},m_0)$ and a non-zero morphism $\M\to\M_i^{p^e}(q_{i,e}bL)$.
By considering the degree of the line bundles, we see that $\M_i^{p^e}(q_{i,e}bL)\cong\M\in T^{p^e}$, where
$$\textstyle T^n:=\{\bigotimes_{1\le i\le\lambda}\M_i^{n_i}\in\Pic(Z)|n_i\ge0,~\sum_{1\le i\le\lambda}n_i=n\}.$$
Fix an integer $e >0$ such that $a|p^e-1$. Set $n:=\lambda(p^e-1)+1$. For every $\mathcal N\in T^n$, there exist $n_1,\ldots,n_\lambda\ge0$ such that
$\mathcal N\cong \bigotimes_{1\le i\le \lambda}\M_i^{n_i}$ and $n'_j:=n_j-p^e\ge0$ for at least one $j$.
Then \begin{align*}
\mathcal N(q_{j,e}bL)\cong(\bigotimes_{i\ne j}\M_i^{n_i})\otimes\M_j^{n'_j}\otimes\M_j^{p^e}(q_{j,e}bL).
\end{align*}
Since $\M_j^{p^e}(q_{j,e}bL)\in T^{p^e}$, we have $\mathcal N(q_{j,e}bL)\in T^{n}$.
Hence for every $m\ge q:=\max\{q_{1,e},\ldots,q_{\lambda,e}\}$,
$$\textstyle \mathcal N(mbL)\in \{\M(k bL)\in\Pic(Z)|\M\in T^n,~0\le k<q\}. $$
Since $T^n$ is a finite set, there are integers $m>m'>0$ such that $\mathcal N(mbL)\cong\mathcal N(m'bL)$, and hence $(m-m')bL\sim0$.
\end{proof}
\begin{proof}[Proof of Theorem \ref{thm:c31_intro}: the case $\kappa(X_{\ol\eta})=0$]
As in the proof of Theorem \ref{mthp}, we may assume that $X$ is minimal over $Z$ and $K_{X_{\ol\eta}}$ is semi-ample, thus $K_{X_{\ol\eta}}\sim_{\Q}0$. By Lemma \ref{l-linear-pullback}, $K_X$ is $\Q$-linearly equivalent to the pullback of $K_Z+L$ for some $\Q$-divisor $L$ on $Z$. In particular $\kappa(X,K_X)=\kappa(Z,K_Z+L)$. It is enough to show that $\kappa(Z,K_Z+L)\ge\kappa(Z)$.
By Lemma \ref{thm:ps-eff_intro}, we see that $L$ is nef. Note that since $X_{\ol\eta}$ has at most rational double points as singularities, $X_{\ol\eta}$ is Gorenstein and  $p>5$, $X_{\ol\eta}$ is $F$-pure by \cite[Section 3]{Art77} and \cite{Fed83}.
When $Z$ is of general type, by Theorem \ref{thm:ps-eff_intro}, we have $K_Z+L$ is big, thus $\kappa(Z,K_Z+L)=\dim Z=\kappa(Z)$.
And when $Z$ is an elliptic curve, by Theorem \ref{thm:sa_intro}, we have $\kappa(Z,K_Z+L)\ge\kappa(Z)$. This completes the proof.
\end{proof}
\section{The case $\kappa(X_{\ol\eta})=1$}\label{section:k(F)=1}
In this section, we consider the case when the Kodaira dimension of the geometric generic fiber is one.
\begin{proof}[Proof of Theorem $\ref{thm:c31_intro}$: the case $\kappa(X_{\ol\eta})=1$]
Let $f:X\to Z$ be a surjective morphism from a smooth projective 3-fold to a smooth projective curve of genus at least one, and let $\ol\eta$ be the geometric generic point of $Z$. Suppose that $\kappa(X_{\ol\eta})=1$.
With the loss of smoothness, by Theorem \ref{thm:mmp} (2) we may assume that $X$ is a minimal model. Then $X_{\ol\eta}$ has canonical singularities by the proof of Theorem \ref{mthp}.

If $g(Z) >1$, then since $f_*\omega_{X/Z}^m$ contains a nef sub-bundle of rank $\geq cm$ for some $c>0$ and any sufficiently divisible $m$ (Theorem \ref{mthp}), by some standard arguments (proof of \cite[Proposition 5.1]{BCZ15}),
we can conclude that
$$\kappa(X) \geq 2 = \kappa(Z) + \kappa(X_{\ol\eta}).$$

So from now on, we assume $g(Z) = 1$. Then $\omega_{X} = \omega_{X/Z}$. We break the proof into several steps.

Step 1: By considering the relative Iitaka fibration and applying Lemma \ref{l-linear-pullback}, we get the following commutative diagram
$$\centerline{\xymatrix{ &X'\ar[d]_h\ar[r]^{\sigma} &X\ar[d]^f  \\ &Y\ar[r]^g  &Z}}$$ where $Y$ is a smooth projective surface, and $h$ is fibration with geometric fiber being a smooth elliptic curve by the proof of Theorem \ref{mthp}, such that $\sigma^*K_X \sim_{\mathbb{Q}} h^*D$ where $D$ is a nef $g$-big divisor on $Y$.

If $D$ is big, then we are done. From now on, we assume the numerical dimension $\nu(K_X)=\nu(D) = 1$. Then we claim that
\begin{cl}\label{cl}
If $X$ has a fibration $f': X \rightarrow W$ to a normal projective curve $W$ such that $K_{F'}$ is numerically trivial, where $F'$ denotes the generic fiber of $f'$.  Assume moreover that there exist $L \in \Pic^0(Z)$ and an integer $m>0$ such that $h^0(X, mK_X + f^*L) >0$.  Then $K_X$ is semi-ample.
\end{cl}
\begin{proof}[Proof of the claim]
Take an effective divisor $D_L \sim mK_X + f^*L$.
Since $D_L$ is nef, effective and $D_L|_{F'} \sim_{num}0$, we have
$$(mK_X + f^*L)|_{F'} \sim D_L|_{F'} \sim 0.$$
By Lemma \ref{l-linear-pullback} we can assume $D_L \sim_{\mathbb{Q}} f'^*A$ where $A$ is a divisor on $W$, which is ample since $D_L\ne0$.
So we only need to show that $L \sim_{\mathbb{Q}} 0$.

Since $X$ has at most terminal singularities, it is smooth in codimension one, so $F'$ is a regular surface over the function filed $K(W)$ of $W$. Applying \cite[Theorem~1.1]{Tan15}, since $K_{F'}$ is numerically trivial,
we have $K_{F'}\sim_{\mathbb{Q}} 0$.
Therefore, we conclude that
$$f^*L|_{F'} \sim_{\Q} mK_{F'} + f^*L|_{F'} \sim_{\Q} (mK_X + f^*L)|_{F'} \sim_{\Q} D_L|_{F'} \sim_{\Q} 0.$$
On the other hand, $F'$ is dominant over the curve $Z\otimes_k K(W)$, passing to the algebraic closure of $K(W)$ and applying Lemma \ref{injofpic},
we show that $L$ is torsion.
\end{proof}
\noindent%

Step 2: By Theorem \ref{mthp}, there exists $c>0$ such that for sufficiently divisible $m_1$, $f_*\omega_{X}^{m_1}$ contains a nef sub-bundle $V$ of rank $r_{m_1} \ge cm_1$. If $\deg V >0$, then we are done by some standard arguments (\cite[Propostion 5.1]{BCZ15}).
So we assume that $\deg V =0$, thus by Proposition \ref{prop:decomp} there exists a flat base change between two elliptic curves $\pi: Z_1 \rightarrow Z$ such that $\pi^*V = \oplus_{i=1}^n \mathcal{L}_i$ where $\mathcal{L}_i \in \mathrm{Pic}^0(Z_1)$. Let $X_1$ be the normalization of $X\times_Z Z_1$. Then we get the following commutative diagram
$$\centerline{\xymatrix{ &X_1\ar[d]_{f_1}\ar[r]^{\pi_1} &X\ar[d]^f  \\ &Z_1\ar[r]^{\pi}  &Z}}$$
where $\pi_1$ and $f_1$ denote the natural projections. We have that $\pi^*f_*\omega_{X}^{m_1} \subset f_{1*}\pi_1^*\omega_{X}^{m_1}$ by \cite[Proposition 9.3]{Ha77}, thus
$$\pi^*V = \oplus_{i=1}^n \mathcal{L}_i \subset f_{1*}\pi_1^*\omega_{X}^{m_1}.$$
So we conclude that
$$h^0(X_1, \pi_1^*m_1K_X - f_1^*\mathcal{L}_i) \geq 1,$$
and if $\mathcal{L}_i = \mathcal{L}_j$ for some $j \neq i$ then the strict inequality holds.
Since $\pi^*: \mathrm{Pic}^0(Z) \to \mathrm{Pic}^0(Z_1)$ is surjective, there exists $L_i'$ such that $\mathcal{L}_i \sim \pi^*L_i'$, thus
$$\pi_1^*m_1K_X - f_1^*\mathcal{L}_i \sim \pi_1^*(m_1K_X + f^*L_i').$$
Applying Theorem \ref{ct}, we can find a sufficiently divisible integer $l >0$ such that
$$h^0(X, l(m_1K_X - f^*L_i')) \geq 1.$$
Put $m = lm_1$ and $L_i = lL_i'$. Then $h^0(X, mK_X - f^*L_i) \geq 1$.

If $h^0(X, mK_X - f^*L_i) > 1$, then $h^0(Y, mD - g^*L_i) >1$ by the construction in Step 1. Since $mD - g^*L_i$ is nef and $\nu(mD - g^*L_i) = 1$, the movable part of $|mD - g^*L_i|$ has no base point, hence induces a fibration $g': Y \to W'$ on $Y$ to a curve $W'$. The Stein factorization of the composite morphism $g' \circ h: X' \to W'$ induces a fibration $f'': X' \to W$ from $X'$ to a normal curve $W$, which is defined by the base point free linear system $|\mu^*l(mK_X - f^*L_i)|$ for sufficiently divisible integer $l>0$. Since $\sigma: X' \to X$ is a birational morphism such that $\sigma_*\O_{X'} = \O_X$, we conclude that $|\mu^*l(mK_X - f^*L_i)| = \mu^*|l(mK_X - f^*L_i)|$, thus $|l(mK_X - f^*L_i)|$ has no base point, hence defines
such a fibration $f': X \to W$ as in Claim of Step 1. So $K_X$ is semi-ample, and this completes the proof in this case.

From now on, we can assume $h^0(X, mK_X - f^*L_i) = 1$ and $h^0(X_1, \pi_1^*(mK_X - f^*L_i)) = 1$. For every $i$, we have unique effective divisors $B_i \sim mK_X - f^*L_i$. And by construction, we can assume $\pi_1^*B_i \neq \pi_1^*B_j$ if $i \neq j$, thus $L_i \neq L_j$.
In the following, we only need to show that at least two of $L_i$ are torsion.\\

Step 3: For every $j$, we have unique effective divisor $B_j \sim mK_X - f^*L_j$.
Let $B'$ be the reduced divisor supported on the union of components of $\sum_j B_j$.
Take a smooth log resolution $\mu: \tilde{X} \rightarrow X$ of the pair $(X, B')$. Denote by $\tilde{f}: \tilde{X} \rightarrow Z$ the natural morphism.
Let $\tilde{B}$ be the reduced divisor supported on the union components of $\sum_j \mu^*B_j$.
Consider the dlt pair $(\tilde{X}, \tilde{B})$.
Since $X$ has terminal singularities, there exists an effective $\mu$-exceptional divisor $E$ on $\tilde{X}$ such that
$$K_{\tilde{X}}  \sim_{\Q} \mu^*K_X + E.$$
So $K_{\tilde{X}} + \tilde{B} \sim_{\Q} \mu^*K_X + E + \tilde{B}$
has a weak Zariski decomposition.
By Theorem \ref{thm:mmp} (2), $(\tilde{X}, \tilde{B})$ has a minimal model $(\hat{X}, \hat{B})$ which is dlt, and there exists a natural morphism $\hat{f}: \hat{X} \to Z$.
By the construction, we have the following:\\
(1) Note that $B_j|_{X_{\bar{\eta}}}$ is contained in finitely many fibers of the Iitaka fibration $I_{\bar{\eta}}: X_{\bar{\eta}} \to C_{\bar{\eta}}$, which implies that $\kappa(\tilde{X}_{\bar{\eta}}, (K_{\tilde{X}} + \tilde{B})|_{X_{\bar{\eta}}}) = 1$.
Since the restriction $(K_{\hat{X}} + \hat{B})|_{\hat{X}_{\bar{\eta}}}$ is semi-ample by \cite[Theorem 1.2]{Tan14}, it induces an elliptic fibration on $\hat{X}_{\bar{\eta}}$ by construction.
So applying Lemma \ref{l-linear-pullback} again, we get the following commutative diagram
$$\centerline{\xymatrix{
&\hat{X}'\ar[d]_{\hat{h}}\ar[r]^{\hat{\sigma}} &\hat{X}\ar[d]^{\hat{f}}  \\
&\hat{Y}\ar[r]^{\hat{g}} &Z
}}$$
where $\hat{Y}$ is a smooth projective surface and $\hat{h}$ is an elliptic fibration such that, $\hat{\sigma}^*(K_{\hat{X}} + \hat{B}) \sim_{\mathbb{Q}} \hat{h}^*\hat{D}$ where $\hat{D}$ is a nef and $\hat{g}$-big divisor on $\hat{Y}$.\\
(2)We claim that $\nu(K_{\hat{X}} + \hat{B}) = \nu(\hat{D}) = 1$. Indeed, otherwise, $\hat{D}$ will be big. Note that the divisor $\mu^*\sum_j B_j  - \tilde{B}$ is effective and $\mu^*\sum_j B_j \sim \mu^*nmK_X - \sum_j\tilde{f}^*L_j$. Then applying Theorem \ref{ct} we can get a contradiction as follows
\begin{equation*}
\begin{split}
2 = &\kappa(\hat{Y}, \hat{D}) =\kappa(\hat{X}', \hat{\sigma}^*(K_{\hat{X}} + \hat{B})) = \kappa(\hat{X}, K_{\hat{X}} + \hat{B}) = \kappa(\tilde{X}, K_{\tilde{X}} + \tilde{B})\\
\leq &\kappa(\tilde{X}, K_{\tilde{X}} + \mu^*nmK_X - \sum_j\tilde{f}^*L_j)\\
= &\kappa(\tilde{X}, \mu^*K_X + E + \mu^*nmK_X - \sum_j\mu^*f^*L_j) \\
= &\kappa(X, (nm+ 1)K_X - \sum_jf^*L_j) = 1.
\end{split}
\end{equation*}
(3) For sufficiently divisible $M$ and every $1 \leq i \leq n$, we get an effective Cartier divisor
$$\tilde{\Gamma}_i = M(\mu^*B_i + mE) + Mm\tilde{B}  \sim M(mK_{\tilde{X}} - \tilde{f}^*L_i) + Mm\tilde{B}.$$
Denote by $\nu: \tilde{X} \dashrightarrow \hat{X}$ the natural birational map. Let $\hat{\Gamma}_i = \nu_*\tilde{\Gamma}_i$. Then
$$\hat{\Gamma}_i \sim M(mK_{\hat{X}} - \hat{f}^*L_i) + Mm\hat{B} \sim Mm(K_{\hat{X}} + \hat{B}) - M\hat{f}^*L_i.$$
Since $E$ is contained in finitely many fibers of $\tilde{f}$, $\nu_*E$ is contracted by $\hat{f}$. So if a component of $\hat{\Gamma}_i$ is dominant over $Z$ then it is contained in $\hat{B}$.\\
(4) Take an effective divisor $\hat{D}_i \sim Mm\hat{D} - M\hat{g}^*L_i$ for each $i$.
Since $\hat{D}$ is nef and $\hat{D}^2 =0$, so is $\hat D_i$.
Considering connected components of the union of the $\hat D_i, 0 \leq i \leq n$, we see that there exist nef effective Cartier divisors $\hat D_1',\ldots,\hat D_k'$ satisfying the conditions below:
\begin{itemize}
\item $\mathrm{Supp}(\hat{D}_j')$ is connected for each $j$, and $\mathrm{Supp}(\hat{D}_j')\cap\mathrm{Supp}(\hat{D}_l')=\emptyset$ for each $j\neq l$;
\item $(\hat{D}_1')^2=\cdots=(\hat{D}_k')^2=0$;
\item the greatest common divisor of the coefficients of every $\hat{D}_j'$ is equal to one;
\item for each $i$, there exist $a_{i1},\ldots,a_{ik} \in \mathbb{Z}^{\geq 0}$ such that $\hat{D}_i = a_{i1}\hat{D}_1' + \cdots + a_{ik}\hat{D}_k'$.
\end{itemize}
Note that at least one of the $\hat{D}_j'$ is dominant over $Z$, and hence intersects every fiber of $\hat{f}$.
From this we see that every $\hat{D}_j'$ is dominant over $Z$. Indeed, if a $\hat D_j'$ is contained in one fiber, then the support of $\hat D_j'$ is equal to the whole fiber as shown by \cite[VIII.4]{Beau96}, which contradicts to the first condition above.
Now we have $\hat{\sigma}^*\hat{\Gamma}_i = \hat{h}^*\hat{D}_i$ by the construction.
Hence $\hat h^*\hat D_1',\ldots,\hat h^*\hat D_k'$ are disjoint connected components of $\hat \sigma^*(\sum \hat\Gamma_i)$.
Let $\hat G_j:=\hat\sigma_*\hat h^*\hat D_j'$. Then we have that $\mathrm{Supp}(\hat G_j)\cap\mathrm{Supp}(\hat G_l)=\emptyset$ for each $j\neq l$. \\
(5)
Take two divisors $\hat{D}_1, \hat{D}_2$. Since $\hat{D}_1\ne\hat D_2$, we may assume that $a_{11}>a_{21}\geq0$.
We may further assume that $a_{22}>a_{12}\geq0$ because $\hat{D}_1 \sim_{num} \hat{D}_2$.
We can write that
$$\hat\Gamma_1 = a_{11}\hat{G}_1 + a_{12}\hat{G}_2 + \hat{G}'_3~\mathrm{and}~ \hat\Gamma_2 = a_{21}\hat{G}_1 + a_{22}\hat{G}_2 + \hat{G}''_3$$
where neither of $\hat{G}_1$ and $\hat{G}_2$ intersects $\hat{G}'_3 \cup \hat{G}''_3$. \\

Step 4: Take two reduced, irreducible and dominant over $Z$ components $\hat{C}_1, \hat{C}_2$ of $\hat{G}_1, \hat{G}_2$ respectively.
Then $\hat{C}_1, \hat{C}_2$ are contained in $\hat{B}$ by the construction of $\hat\Gamma_i$ in Step 3 (3). Since $(\hat{X}, \hat{B})$ is dlt and $\hat{B}$ is a reduced integral divisor, so $\hat{C}_1, \hat{C}_2$ are log canonical centers of $(\hat{X}, \hat{B})$. By Step 3 (4), since $\hat{D}$ is nef and $\hat{D}\cdot\hat{D}_i = 0$, so $\hat{D}|_{\hat{D}_i} \sim_{num} 0$. For $j=1,2$, since $\hat{h}(\hat{\sigma}^{-1}\hat{C}_j)$ is a component of some $\hat{D}_i$, by
$\hat{\sigma}^*(K_{\hat{X}} + \hat{B}) \sim_{\Q} \hat{h}^*\hat{D}$, we conclude that
$$\hat{\sigma}^*(K_{\hat{X}} + \hat{B})|_{\hat{\sigma}^{-1}\hat{C}_j} \sim_{num} 0.$$
Denote by $\hat{C}'_i$ the normalization of $\hat{C}_i$. Then $(K_{\hat{X}} + \hat{B})|_{\hat{C}'_i} \sim_{num} 0$, thus
$(K_{\hat{X}} + \hat{B})|_{\hat{C}'_i} \sim_{\mathbb{Q}} 0$ by Lemma \ref{adj}.
Therefore,
\begin{equation}
\begin{split}
-a_{21}M\hat{f}^*L_1|_{\hat{C}'_1} &\sim_{\mathbb{Q}} a_{21}(Mm(K_{\hat{X}} + \hat{B}) - M\hat{f}^*L_1)|_{\hat{C}'_1}\\
&\sim_{\mathbb{Q}} a_{21}\hat{\Gamma}_1|_{\hat{C}'_1} \sim_{\mathbb{Q}} a_{11}a_{21}\hat{G}_1|_{\hat{C}'_1}\\
&\sim_{\mathbb{Q}} a_{11}\hat{\Gamma}_2|_{\hat{C}'_1} \sim_{\mathbb{Q}} -a_{11}M\hat{f}^*L_2|_{\hat{C}'_1}
\end{split}
\end{equation}
which, by Lemma \ref{injofpic}, implies that
$$a_{21}ML_1 \sim_{\mathbb{Q}} a_{11}ML_2.$$
In the same way, restricting on $\hat{C}'_2$ gives
$$a_{22}ML_1 \sim_{\mathbb{Q}} a_{12}ML_2.$$
Finally by conditions $a_{11} > a_{21}$ and $a_{12}<a_{22}$, we conclude that $L_1 \sim_{\mathbb{Q}} L_2 \sim_{\mathbb{Q}} 0$, and this completes the proof.
\end{proof}
%
%
%


\begin{thebibliography}{99} \small
%
\bibitem{Abh98} S. Abhyankar: {\it Resolution of singularities of embedded algebraic surfaces}, Second edition, Springer Monographs in Mathematics, SpringerVerlag, Berlin, (1998).
%
\bibitem{Ati57} M. F. Atiyah: {\it Vector bundles over an elliptic curve}, Proc. London Math. Soc, \tb{7} (1957), 414--452.
%
\bibitem{Art77} M. Artin: {\it Coverings of the rational double points in characteristics p}, Complex Analysis and Algebraic Geometry, Iwanami Shoten, Tokyo, (1977).
%
\bibitem{Bad01} L. B\v{a}descu: {\it Algebraic surfaces}, Universitext, Springer-Verlag, NewYork, (2001).
%
\bibitem{Beau96} A. Beauville: {\it Complex algebraic surfaces}, second edition. London Mathematical Society Student Texts. 34, Cambridge University Press (1996).
%
\bibitem{Bir09} C. Birkar: {\it The Iitaka conjecture $C_{n,m}$ in dimension six}, Compositio Math. \tb{145} (2009), no. 6, 1442--1446.
%
\bibitem{Bir13} C. Birkar: {\it Existence of flips and minimal models for 3-folds in char $p$}, Ann. Sci. \'{E}c. Norm. Sup\'{e}r. \tb{49} (2016), no. 1, 169--212.
%
\bibitem{BW14} C. Birkar, J. Waldron: {\it Existence of Mori fibre spaces for 3-folds in char $p$}, Adv. Math. \tb{313} (2017), 62--101. 
%
\bibitem{BCZ15} C. Birkar, Y. Chen, L. Zhang: {\it Iitaka's $C_{n,m}$ conjecture for 3-folds over finite fields}, Nagoya Math. J. \tb{229} (2018), 21--51.
%
\bibitem{BM2} E. Bombieri, D. Mumford: {\it Enriques' classification of surfaces in char.p, II}, In Complex Analysis and Algebraic Geometry (dedicated to K. Kodaira). Iwanami Shoten Publ., Tokyo, Cambridge Univ. (1977), Part I, 23--42.
%
\bibitem{Cao15} J. Cao: {\it Kodaira dimension of algebraic fiber spaces over surfaces}, to appear in Algebraic Geometry, http://arxiv.org/abs/1511.07048 (2015).
%
\bibitem{CP15} J. Cao, M. P\v aun: {\it Kodaira dimension of algebraic fiber spaces over Abelian varieties}, Invent. Math. \tb{207} (2017), no. 1, 345--387.
%
\bibitem{CTX15} P. Cascini, H. Tanaka, C. Xu: {\it On base point freeness in positive characteristic}, Ann. Sci. \'{E}c. Norm. Sup\'{e}r.  \tb{48} (2015), 1239--1272.
%
\bibitem{CH11} J.A. Chen, C.D. Hacon: {\it Kodaira dimension of irregular varieties}, Invent. Math. \tb{186} (2011), no. 3, 481--500.
%
\bibitem{CZ13} Y. Chen, L. Zhang: {\it The subadditivity of the Kodaira dimension for fibrations of relative dimension one in positive characteristics}, Math. Res. Lett. 22 (2015), 675--696.
%
\bibitem{CP08} V. Cossart, O. Piltant: {\it Resolution of singularities of threefolds in positive characteristic I}, J. Algebra \tb{320} (2008), 1051--1082.
%
\bibitem{CP09} V. Cossart, O. Piltant: {\it Resolution of singularities of threefolds in positive characteristic II}, J. Algebra \tb{321} (2009),1836--1976.
%
\bibitem{Cut04} S. D. Cutkosky: {\it Resolution of singularities for 3-folds in positive characteristic}, Amer. J. Math. \tb{131} (2009), no. 1, 59--127.
%
\bibitem{Eji15} S. Ejiri: {\it Weak positivity theorem and Frobenius stable canonical rings of geometric generic fibers}, J. Algebraic Geom. \tb{26} (2017), 691--734.
%
\bibitem{Eji16} S. Ejiri: {\it Positivity of anti-canonical divisors and $F$-purity of fibers}, https://arxiv.org/abs/1604.02022. 
%
\bibitem{Fed83} R. Fedder: {\it F-purity and rational singularity}, Trans. Amer. Math. Soc. \tb{278} (1983), no. 2, 461--480.
%
\bibitem{Fuj13} O. Fujino: {\it On maximal Albanese dimensional varieties}, Proc. Japan Acad. Ser. A Math. Sci. \tb{89} (2013), no. 8, 92--95.
%
\bibitem{Fuj14} O. Fujino: {\it On subadditivity of the logarithmic Kodaira dimension}, J. Math. Soc. Japan, \tb{69}, (2017), no. 4, 1565--1581.
%
\bibitem{HX15} C.D. Hacon, C. Xu: {\it On the three dimensional minimal model program in positive characteristic}, J. Amer. Math. Soc. \tb{28} (2015), 711--744.
%
\bibitem{Ha77} R. Hartshorne: \textit{Algebraic Geometry}, Graduate Texts in Mathematics, No. 52. 1977.
%
\bibitem{Iit82} S. Iitaka: \textit{Algebraic geometry}, An introduction to birational geometry of algebraic varieties,
Graduate Texts in Mathematics, 76. 1982.
%
\bibitem{Kaw81} Y. Kawamata: {\it Characterization of abelian varieties}, Compositio Math. \tb{43} (1981), no. 2, 253--276.
%
\bibitem{Kaw82} Y. Kawamata: {\it Kodaira dimension of algebraic fiber spaces over curves}, Invent. Math. \tb{66} (1982), no. 1, 57--71.
%
\bibitem{Kaw85} Y. Kawamata: {\it Minimal models and the Kodaira dimension of algebraic fiber spaces}, J. Reine Angew. Math. \tb{363} (1985), 1--46.
%
\bibitem{Ke99} S. Keel: \textit{Basepoint freeness for nef and big line bundles in positive characteristic}, Annals of Math., Second Series, \tb{149}, no. 1 (1999), 253--286.
%
\bibitem{Kol87} J. Koll\'ar: {\it Subadditivity of the Kodaira dimension: fibers of general type}, Algebraic geometry, Sendai, 1985, 361--398,  Adv. Stud. Pure Math. \tb{10}, North-Holland, Amsterdam, (1987).
%
\bibitem{KP15} S. Kov\'acs, Z. Patakfalvi: {\it Projectivity of the moduli space of stable log-varieties and subadditivity of log-Kodaira dimension}, J. Amer. Math. Soc. \tb{30} (2017), 959--1021.
%
\bibitem{LS77} H. Lange, U. Stuhler: {\it Vektorb\"undel auf Kurven und Darstellungen Fundamentalgruppe}, Math Z. \tb{156} (1977), 73--83.
%
\bibitem{Lai11} C. Lai: {\it Varieties fibered by good minimal models}, Math. Ann. \tb{350} (2011), no. 3, 533--547.
%
\bibitem{Oda71} T. Oda: {\it Vector bundles on an elliptic curve}, Nagoya Math. J. \tb{43} (1971), 41--72.
%
\bibitem{Pat14} Z. Patakfalvi: {\it Semi-positivity in positive characteristics}, Ann. Sci. \'{E}c. Norm. Sup\'{e}r.  \tb{47} (2014), no. 5, 991--1025.
%
\bibitem{Pat13} Z. Patakfalvi: {\it On subadditivity of Kodaira dimension in positive characteristic over a general type base}, J. Algebraic Geom. \tb{27} (2018), 21--53. 
%
\bibitem{Tan14} H. Tanaka: {\it Minimal models and abundance for positive characteristic log surfaces}, Nagoya Math. J. \tb{216} (2014), 1--70.
%
\bibitem{Tan15x} H. Tanaka: {\it The X-method for klt surfaces in positive characteristic}, J. Algebraic Geom. \tb{24} (2015), 605--628.
%
\bibitem{Tan15} H. Tanaka: {\it Abundance theorem for surfaces over an imperfect field}, http://arxiv.org/abs/1502.01383 (2015).
%
\bibitem{Vie77} E. Viehweg: {\it Canonical divisors and the additivity of the Kodaira dimension for morphisms of relative dimension one}, Compositio Math. \tb{35} (1977), no. 2, 197--223.
%
\bibitem{Vie83} E. Viehweg: {\it Weak positivity and the additivity of the Kodaira dimension for certain fibre spaces}, Algebraic varieties and analytic varieties (Tokyo, 1981), 329--353, Adv. Stud. Pure Math. \tb{1} North-Holland, Amsterdam, (1983).

\bibitem{Wal15} J. Waldron: {\it Finite generation of the log canonical ring for 3-folds in char $p$}, Math. Res. Lett. \tb{24} (2017), no. 3, 933--946. 
%
\bibitem{Wal16} J. Waldron: {\it The LMMP for log canonical 3-folds in characteristic $p>5$}, Nagoya Math. J. \tb{230} (2018), 48--71. 
%
\bibitem{Zha16} L. Zhang: {\it Subadditivity of Kodaira dimensions for fibrations of three-folds in positive characteristics}, http://arxiv.org/abs/1601.06907 (2016).
%
\end{thebibliography}
\end{document}